\documentclass[11pt,reqno]{amsart}

\usepackage{amsmath, amsthm, amssymb}
\usepackage{amsfonts}
\usepackage[ansinew]{inputenc}
\usepackage[dvips]{epsfig}
\usepackage{graphicx}
\usepackage[english]{babel}
\usepackage{comment}
\usepackage{graphicx}
\usepackage{hyperref}
\usepackage{bbm} 

\usepackage{cite}
\usepackage{graphicx}
\usepackage{amscd}
\usepackage{xcolor}
\usepackage{bm}
\usepackage{enumerate}

\usepackage{verbatim}
\usepackage{hyperref}
\usepackage{amstext}
\usepackage{latexsym}
\let\oldsqrt\sqrt
\def\sqrt{\mathpalette\DHLhksqrt}
\def\DHLhksqrt#1#2{%
\setbox0=\hbox{$#1\oldsqrt{#2\,}$}\dimen0=\ht0
\advance\dimen0-0.2\ht0
\setbox2=\hbox{\vrule height\ht0 depth -\dimen0}%
{\box0\lower0.4pt\box2}}

\allowdisplaybreaks

\newcommand{\R}{\mathbb{R}} 

\newcommand{\dist}{\textnormal{dist}} 
\newcommand{\supp}{\textnormal{supp}} 

\renewcommand{\phi}{\varphi}

\newcommand{\cB}{{\mathcal B}}

\theoremstyle{definition}
\newtheorem{defi}{Definition}[section]

\theoremstyle{plain} 
\newtheorem{thm}[defi]{Theorem}

\newtheorem{lemma}[defi]{Lemma}

\theoremstyle{definition}

\numberwithin{equation}{section}

 \title[A note on P\'{o}lya-Szeg\"{o} inequality for fractional Orlicz-Sobolev seminorm in domains]{A note on P\'{o}lya-Szeg\"{o} inequality for fractional Orlicz-Sobolev seminorm in domains}  

\author[Remi Yvant Temgoua]{Remi Yvant Temgoua}

\address{The Fields Institute for research in mathematical sciences, 222 college street, 2nd floor, Toronto, Ontario, m5t 3j1 Canada, and School of Mathematics and Statistics, Carleton University, 1125 Colonel By Dr, Ottawa, Ontario, K1S 5B6, Canada, and University of Bertoua, P.O Box: 416 Bertoua, Cameroon} 

\email{rtemgoua@fieldsinstitute.ca}

\date{\today
	}

\begin{document}

\begin{abstract}
	In this paper, we study the effect of symmetric radial decreasing rearrangement on fractional Orlicz-Sobolev seminorm in domains. Roughly speaking, we prove that symmetric radial decreasing rearrangement can increase the fractional Orlicz-Sobolev seminorm in domains. Our result extends that of Li-Wang [\textit{Commun. Contemp. Math. 21.07 (2019): 1850059.}] to the setting of fractional seminorm in domains admitting behaviors more general than powers. 
\end{abstract}

\maketitle

	
		
		

\section{Introduction and main results}\label{section:introduction}
P\'{o}lya-Szeg\"{o} inequality plays an important role in many questions from mathematical physics, spectral theory, and analysis of linear and non-linear PDE. In particular, it has application in the study of isoperimetric inequalities, in the Faber-Krahn inequality and in establishing optimal  geometric and functional inequalities such as the optimal Sobolev inequality, Hardy-Littlewood-Sobolev inequality, and Moser-Trudinger inequality, see for instance \cite{frank2008non,talenti1976best,talenti1994inequalities,lieb1983sharp}. Recall that the P\'{o}lya-Szeg\"{o} inequality \cite{polya1951isoperimetric} states that: if $u\in W^{1,p}(\R^N)$ is nonnegative, then $u^*\in W^{1,p}(\R^N)$ and
\begin{equation*}
\int_{\R^N}|\nabla u^*|^p\ dx\leq \int_{\R^N}|\nabla u|^p\ dx,
\end{equation*}
where $u^*$ is the symmetric decreasing rearrangement or Schwarz symmetrization of $u$. This inequality was extended to nonlocal setting in \cite{almgren1989symmetric}, where the authors proved that 
\begin{equation*}
\iint_{\R^N\times\R^N}\frac{|u^*(x)-u^*(y)|^p}{|x-y|^{N+sp}}\ dxdy\leq\iint_{\R^N\times\R^N}\frac{|u(x)-u(y)|^p}{|x-y|^{N+sp}}\ dxdy
\end{equation*}
for all $u\in W^{s,p}(\R^N)$ with $p>1$ and $s\in(0,1)$. Recently, P\'{o}lya-Szeg\"{o} inequality was extended to the  more general setting of fractional Orlicz-Sobolev spaces in \cite{de2021polya}. More precisely, let $G:\R\to\R$ be a Young function satisfying the growth condition
\begin{equation}\label{g}
1<p^-_{G}\leq\frac{tg(t)}{G(t)}\leq p^+_{G}<\infty~~~\forall t>0,
\end{equation}
where $g=G', p^{-}_G:=\inf_{t>0}\frac{tg(t)}{G(t)}$ and $p^+_G:=\sup_{t>0}\frac{tg(t)}{G(t)}.$ Consider two functions $M, N: \R_+\to\R_+$ satisfying
\begin{align}
& M~\text{and}~N~\text{are nondecreasing and}~M(r), N(r)>0~ \text{for}~r>0,  \label{m1}\\
& M(r)\geq\min\{1,r\}~\text{and}~N~\text{is contibuous},   \label{m2}\\
& \int_{0}^{1}\frac{r^{N-1+p^-_G}}{N(r)M(r)^{p^-_G}}\ dr<\infty~\text{and}~\int_{1}^{\infty}\frac{r^{N-1}}{N(r)M(r)^{p^-_G}}\ dr<\infty.  \label{m3}
\end{align}
Then the authors in \cite{de2021polya} proved that under the assumptions \eqref{g}, \eqref{m1}, \eqref{m2}, and \eqref{m3}, the following inequality 
\begin{equation}\label{polya-szego-orlicz-rn} 
\iint_{\R^N\times\R^N}G\Big(\frac{u^*(x)-u^*(y)}{M(|x-y|)}\Big)\frac{dxdy}{N(|x-y|)}\leq\iint_{\R^N\times\R^N}G\Big(\frac{u(x)-u(y)}{M(|x-y|)}\Big)\frac{dxdy}{N(|x-y|)} 
\end{equation}
holds for all $u\in W^{M,N,G}(\R^N)$. Here, $W^{M,N,G}(\R^N)$ consists of measurable functions $u: \R^N\to\R$ such that $\int_{\R^N}G(u)\ dx<\infty$ and $\iint_{\R^N\times\R^N}G\Big(\frac{u(x)-u(y)}{M(|x-y|)}\Big)\frac{dxdy}{N(|x-y|)}<\infty.$ Moreover, as a direct application of \eqref{polya-szego-orlicz-rn}, a Faber-Krahn type inequality for Dirichlet g-eigenvalues and Poincar\'{e}'s constants for the fractional $g$-Laplacian is also proved in \cite{de2021polya}.

Motivated by the above discussion, we are interested in analyzing whether \eqref{polya-szego-orlicz-rn} remains valid  if $\R^N$ is replaced by a domain $\Omega$ of finite measure on the right-hand side (and correspondingly by the symmetric rearrangement $\Omega^*$ of $\Omega$ on the left-hand side). That is, do we have
\begin{equation}\label{polya-szego-orlicz-domain} 
\iint_{\Omega^*\times\Omega^*}G\Big(\frac{u^*(x)-u^*(y)}{M(|x-y|)}\Big)\frac{dxdy}{N(|x-y|)}\leq\iint_{\Omega\times\Omega}G\Big(\frac{u(x)-u(y)}{M(|x-y|)}\Big)\frac{dxdy}{N(|x-y|)}?
\end{equation}
Another motivation for considering the above question comes from the study of Faber-Krahn type inequality for Dirichlet g-eigenvalues and Poincar\'{e}'s constants for the regional fractional $g$-Laplacian. To be more precise, consider the eigenvalue problem 
\begin{equation}\label{eigenvalue-problem}
\left\{
\begin{aligned}
(-\Delta_g)^s_\Omega u&=\lambda g(|u|)\frac{u}{|u|}~~~\text{in}~\Omega\\
u&=0~~~~~~~~~~~~~~~\text{on}~\partial\Omega,
\end{aligned}
\right.
\end{equation}
where $(-\Delta_g)^s_{\Omega}$ is the so-called regional fractional $g$-Laplacian defined by 
\begin{equation*}
(-\Delta_g)^s_{\Omega}u(x):=P.V.\int_{\Omega}g\Big(\frac{|u(x)-u(y)|}{|x-y|^s}\Big)\frac{u(x)-u(y)}{|u(x)-u(y)|}\frac{dy}{|x-y|^{N+s}}.
\end{equation*}
This operator is well defined and continuous between $W^{s,G}(\Omega)$ and its dual. Here, $W^{s,G}(\Omega)$  consists of measurable functions $u:\R^N\to\R$ such that $\int_{\Omega}G(|u|)\ dx<\infty$ and $\iint_{\Omega\times \Omega}G(\frac{|u(x)-u(y)|}{|x-y|^s})\frac{dxdy}{|x-y|^N}<\infty$. Next, we define 
\begin{equation*}
W^{s,G}_0(\Omega):=\{u\in W^{s,G}(\Omega): u=0~\text{a.e. on}~\partial\Omega\}.
\end{equation*}
Arguing as in \cite{salort2020eigenvalues}, it can be proved that \eqref{eigenvalue-problem} has a sequence of eigenvalues with corresponding eigenfunctions which belong to  $W^{s,G}_0(\Omega)$. Moreover, due to possible lack of homogeneity of \eqref{eigenvalue-problem}, eigenfunctions depend strongly on the energy level. Precisely, for any $\mu>0$ there exists $\lambda_{\mu}>0$ and an eigenfunction $u_{\mu}\in W^{s,G}_0(\Omega)$ normalized such as $\int_{\Omega}G(|u_{\mu}|)\ dx=\mu$, that is,
\begin{align*}
\langle (-\Delta_g)^s_{\Omega}u_{\mu}, \phi\rangle:=\iint_{\Omega\times\Omega}g(|D_su_{\mu}|)\frac{D_su_{\mu}}{|D_su_{\mu}|}D_s\phi\frac{dxdy}{|x-y|^N}=\lambda_{\mu}(\Omega)\int_{\Omega}g(|u_{\mu}|)\frac{u_{\mu}}{|u_{\mu}|}\phi\ dx,
\end{align*}
for all $\phi\in W^{s,G}_0(\Omega)$, where $D_su:=\frac{u(x)-u(y)}{|x-y|^s}$ is the $s$-H\"{o}lder quotient. Note that $\langle\cdot ,\cdot\rangle$ is the duality product. 

Hence, as in \cite{salort2020eigenvalues}, the first eigenvalue of \eqref{eigenvalue-problem} can be defined as
\begin{equation*}
\lambda_1^G(\Omega)=\inf\{\lambda_{\mu}(\Omega): \mu>0\}.
\end{equation*}
We also define
\begin{equation*}
\alpha^G_{\mu}(\Omega)=\inf\Bigg\{\iint_{\Omega\times \Omega}G\Big(\frac{|u(x)-u(y)|}{|x-y|^s}\Big)\frac{dxdy}{|x-y|^N}: u\in W^{s,G}_0(\Omega), \int_{\Omega}G(|u|)\ dx=\mu\Bigg\}  
\end{equation*}
and the best Poincar\'{e} constant over all possible values of $\mu$
\begin{equation*}
\alpha^G_1(\Omega)=\inf\{\alpha^G_{\mu}(\Omega): \mu>0\}.
\end{equation*}
The nontrivial dependence of $\lambda_1^G(\Omega)$ and $\alpha_1^G(\Omega)$ on $\Omega$ leads to interesting and natural questions of shape optimization type: \\\\  
\textbf{Open Question 1.} \textit{Let $G$ be a Young function satisfying \eqref{g}. What are all open sets $\Omega\subset\R^N$ with $\lambda_1^G(\Omega)=\underline{\lambda}^G$, where $\underline{\lambda}^G:=\inf\{\lambda_1^G(V): V\subset\R^N~\text{is an open set}\}>0$?}\\\\
\textbf{Open Question 2.} \textit{Let $G$ be a Young function satisfying \eqref{g}. What are all open sets $\Omega\subset\R^N$ with $\alpha_1^G(\Omega)=\underline{\alpha}^G$, where $\underline{\alpha}^G:=\inf\{\alpha_1^G(V): V\subset\R^N~\text{is an open set}\}>0$?}\\

To the best of the author's knowledge, it is unknown whether there are any such sets. A tempting conjecture to both questions would be that all such sets are balls, which could follow from \eqref{polya-szego-orlicz-domain}. However, we show in this paper that \eqref{polya-szego-orlicz-domain} fails to holds.  More precisely, we have 
\begin{thm}\label{main-result}
	Let $G$ be a Young function satisfying \eqref{g}. Let $N\geq1, s\in(0,1)$ and let $\Omega$ be a nonempty open set of $\R^N$ such that $|\Omega|<\infty$. There exists a nonnegative $u\in C^{\infty}_c(\Omega)$ such that 
	\begin{equation}\label{e1}
	\iint_{\Omega\times\Omega}G\Big(\frac{|u(x)-u(y)|}{|x-y|^s}\Big)\frac{dxdy}{|x-y|^N}<	\iint_{\Omega^*\times\Omega^*}G\Big(\frac{|u^*(x)-u^*(y)|}{|x-y|^s}\Big)\frac{dxdy}{|x-y|^N}. 
	\end{equation}
\end{thm}
It follows from our result that the fractional Orlicz-Sobolev seminorm in domains fails to be decreasing upon Schwarz symmetrization. This failure of P\'{o}lya-Szeg\"{o} principle for fractional Orlicz-Sobolev seminorm in domains $I_{N,s,G,\Omega}[u]:=\iint_{\Omega\times\Omega}G(\frac{|u(x)-u(y)|}{|x-y|^s})\frac{dxdy}{|x-y|^N}$ provides a remarkable difference between the two norms $I_{N,s,G,\Omega}[u]$ and $I_{N,s,G,\R^N}[u]$ on the set $C^{\infty}_c(\Omega)$. Moreover, \eqref{e1} does not support the arguments usually used to prove a Faber-Krahn type inequality for Dirichlet g-eigenvalues and Poincar\'{e}'s constants in contrast with \eqref{polya-szego-orlicz-rn}, see \cite{de2021polya}. 

On the other hand, we have the following estimate. It  states that the converse of \eqref{e1} is true up to a constant depending on $N, s, G$ and $\Omega$. This is our second main result.
\begin{thm}\label{second-main-result}
	Let $G$ be a Young function satisfying \eqref{g}. Let $N\geq1$ and $s\in(0,1)$. Define $\beta_{s,G}:[0,\infty)\to\R$ by
	\begin{equation*}
	\beta_{s,G}(\lambda):=\sup_{t\in[0,\infty)}\frac{G(\lambda t)}{\lambda^{\frac{1}{s}}G(t)}.
	\end{equation*}
	 Let $\Omega$ be a nonempty open set of $\R^N$. Assume that one of the following holds:
	 \begin{itemize}
	 	\item [$(1)$] $\Omega$ is a bounded Lipschitz domain and $\liminf_{\lambda\to0}\beta_{s,G}(\lambda)=0$;
	 	
	 	\item [$(2)$] $\Omega=\{(x',x_N)\in\R^N: x'\in\R^{N-1}, x_N>\Phi(x')\}$, where $\Phi:\R^{N-1}\to\R$ is a Lipschitz map and $s, G$ are such that $\liminf_{\lambda\to0}\beta_{s,G}(\lambda)=0$ or $\liminf_{\lambda\to\infty}\beta_{s,G}(\lambda)=0$; 
	 	
	 	\item [$(3)$] $\R^N\setminus\Omega$ is closure of some bounded Lipschitz domain and $s, G$ are such that $\liminf_{\lambda\to\infty}\lambda^{\frac{1-N}{s}}\beta_{s,G}(\lambda)=0=\liminf_{\lambda\to0}\beta_{s,G}(\lambda)$ or $\liminf_{\lambda\to0}\lambda^{\frac{1-N}{s}}\beta_{s,G}(\lambda)=0$ or $\liminf_{\lambda\to\infty}\beta_{s,G}(\lambda)=0$.
	 \end{itemize} 
	 Then there exists a constant $C=C(N,s,G,\Omega)>0$ such that
	\begin{equation}\label{n1}
	\iint_{\R^N\times\R^N}G\Big(\frac{|u^*(x)-u^*(y)|}{|x-y|^s}\Big)\frac{dxdy}{|x-y|^N}\leq C\iint_{\Omega\times\Omega}G\Big(\frac{|u(x)-u(y)|}{|x-y|^s}\Big)\frac{dxdy}{|x-y|^N}
	\end{equation}
	for all nonnegative $u\in C^{\infty}_c(\Omega)$. In particular,
	\begin{equation}\label{n2}
	\iint_{\Omega^*\times\Omega^*}G\Big(\frac{|u^*(x)-u^*(y)|}{|x-y|^s}\Big)\frac{dxdy}{|x-y|^N}\leq C\iint_{\Omega\times\Omega}G\Big(\frac{|u(x)-u(y)|}{|x-y|^s}\Big)\frac{dxdy}{|x-y|^N}.
	\end{equation}
\end{thm}
Observe that when $G(t)=t^p$ with $p\geq2$, our results coincide with those of \cite{li2019symmetric}. Thus, the results of this paper extend those of \cite{li2019symmetric} to the more general setting of Orlicz-Sobolev spaces, allowing growth laws different than powers  such as $G(t)=t^p(1+|\log t|)$, with $p\geq 2$, or models related to double phase problems where $G(t)=t^q+t^p$ with $p, q\geq2$. 

Our proof of Theorem \ref{main-result} is inspired by the approach of Li-Wang \cite{li2019symmetric}, which is based on an explicit computation of the seminorms. To prove Theorem \ref{second-main-result}, we use the P\'{o}lya-Szeg\"{o} inequality \eqref{polya-szego-orlicz-rn} (with $M(t)=t^s$ and $N(t)=t^N$) together with the fractional Orlicz-Hardy inequality of Bal-Mohanta-Roy-SK \cite{bal2022hardy}.

The paper is organized as follows. In Section \ref{section:preliminary} we collect some definitions and properties of symmetric rearrangements and Young functions. We end this section with some notations. In Section \ref{section:proof-of-main-result} we prove the main results of this paper. In the last section, we present some extensions and applications of our main results.

\section{Preliminary}\label{section:preliminary}
In this section, we recall the notion of symmetric rearrangement of a function and set, as well as we introduce Young functions and present some of their basics properties.

\subsection{Symmetric rearrangement}
Let $A$ be a Borel set in $\R^N$ with $|A|<\infty$ (being $|A|$ the Lebesgue measure of $A$). The symmetric rearrangement $A^*$ of $A$ is defined as
\begin{equation*}
A^*=\Big\{x: |x|<\Big(\frac{|A|}{\alpha_N}\Big)^{\frac{1}{N}}\Big\},
\end{equation*}
where $\alpha_N=\frac{\pi^{\frac{N}{2}}}{\Gamma(\frac{N}{2}+1)}$ is the volume of the unit ball in $\R^N$. Note that if $|A|=0$ then $A^*=\emptyset$.

Denote by $\cB_0$ the space of Borel measurable functions $u:\R^N\to\R_+$ with the property that
\begin{equation*}
|\{x: u(x)>t\}|<\infty~~\text{for all}~t>0.
\end{equation*}
The symmetric decreasing rearrangement or Schwarz symmetrization $u^*$ of a function $u\in\cB_0$ is defined as 
\begin{equation*}
u^*(x)=\int_{0}^{\infty}\chi_{\{|u|>t\}^*}(x)\ dt=\sup\{t: |\{|u|>t\}|>\alpha_N|x|^N\},
\end{equation*}
where $\chi_A$ is the characteristic function of any set $A$. Notice that $u^*$ is radial and radially decreasing function whose sublevel sets have the same measure as those of $u$, namely,
\begin{equation*}
|\{x: u^*(x)>t\}|=|\{x: u(x)>t\}|~~\text{for all}~t>0.
\end{equation*}
Moreover, it follows from the level set characterisation that $u^*$ is inavariant under the action of the translation operator
\begin{equation*}
\tau_zu(x)=u(x-z).
\end{equation*}
In other words, it holds that $(\tau_zu)^*=u^*$. This property will be used throughout the paper and will not be mentioned explicitly.

\subsection{Young functions}
An application $G:\R_+\to\R_+$ is called Young function if it has the integral representation
\begin{equation*}
G(t)=\int_{0}^{t}g(\tau)\ d\tau
\end{equation*}
where the right-continuous function $g:\R_+\to\R_+$ satisfies the following properties
\begin{align*}
	&g(0)=0,~ g(t)>0~\text{for}~t>0\\
	&g~\text{is non-decreasing on}~(0,\infty)\\
	&\lim\limits_{t\to\infty}g(t)=\infty.
\end{align*}
Throughout the paper, the Young function $G$ is assumed to satisfy the growing condition
\begin{equation}\label{g1}
1<p^-_{G}\leq\frac{tg(t)}{G(t)}\leq p^+_{G}<\infty~~~\forall t>0.
\end{equation}
Roughly speaking, condition \eqref{g1} tells that $G$ remains between two power functions. Notice that in \cite[Theorem 4.1]{kranosel1961convex} (see also \cite[Theorem 3.4.4]{kufner1979function}), it is shown that the upper bound in \eqref{g1} is equivalent to the so-called \textit{$\Delta_2$ condition (or doubling condition)}, namely, 
\begin{equation*} 
\quad\quad\quad\quad\quad\qquad g(2t)\leq 2^{q^{+}-1}g(t),\quad\quad\quad G(2t)\leq 2^{q^+}G(t)~~~~~~t\geq0.\qquad\qquad\qquad\qquad\qquad~~~(\Delta_2)     
\end{equation*}
As a consequence of \eqref{g1}, we have the following (see e.g.  \cite[Lemma 2.1]{bahrouni2021neumann}). 
\begin{lemma}\label{two-side-g}
	Let $G$ be a Young function satisfying \eqref{g1}. Then for all $b, a\geq0$,
	\begin{equation*}
	\min\{a^{p^-_G}, a^{p^+_G}\}G(b)\leq G(ab)\leq \max\{a^{p^-_G}, a^{p^+_G}\} G(b). 
	\end{equation*}
\end{lemma}
Given a Young function $G$, its complementary $\tilde{G}$ is defined as
\begin{equation*}
\tilde{G}(t)=\sup\{tw-G(w): w>0\}. 
\end{equation*}
By definition of $\tilde{G}$,
\begin{equation*}
ab\leq G(a)+\tilde{G}(b),~~\text{for all}~~ a, b\geq0.
\end{equation*}
This is called Young's inequality. Moreover, the following identity holds (see \cite[Lemma 2.9]{bonder2019fractional})
\begin{equation*}
\tilde{G}(g(t))=tg(t)-G(t),~~~\forall t>0.
\end{equation*}
Combining this identity with \eqref{g1}, we get
\begin{equation*}
(p^+_G)'\leq\frac{t\tilde{g}(t)}{\tilde{G}(t)}\leq (p^-_G)'~~~\forall t>0, 
\end{equation*}
from where it follows that $\tilde{G}$ also satisfies the $\Delta_2$ condition. Here, $\tilde{g}(t)=(\tilde{G})'(t)$ and $(p^+_G)'=\frac{P^+_G}{P^+_G-1}$ (resp. $(p^-_G)'=\frac{P^-_G}{P^{-}_G-1}$) is the conjugate of $p^+_G$ (resp. $p^-_G$). 

\subsection{Notations} We denote by $A\Delta B=(A\setminus B)\cup(B\setminus A)$ the symmetric difference of any  two sets $A$ and $B$. We will denote by $B_r(x)$ the ball centered at $x$ with radius $r$ in $\R^N$.

\section{Proof of Theorems \ref{main-result} and \ref{second-main-result}}\label{section:proof-of-main-result}

In this section, we prove Theorems \ref{main-result} and \ref{second-main-result}. The following lemma is crucial to the proof of our first main result. 
\begin{lemma}\cite[Lemma 2.1]{li2019symmetric}
	Let  $\Omega$ be an open and bounded set in $\R^N$.
	\begin{itemize}
		\item [$(i)$] Suppose $f\in L^1_{loc}(\R^N)$ is radial and strictly decreasing. Then
		\begin{equation}\label{key-lema1}
		\int_{\Omega^*}f(x)\ dx>\int_{\Omega}f(x)\ dx~~~\text{if}~~|\Omega^*\Delta\Omega|>0.
		\end{equation}
		
		\item [$(ii)$] Suppose $\overline{B}_{\delta}\subset\Omega$ for some $\delta>0$, and let $f\in L^1(\R^N\setminus \overline{B}_{\delta})$ be radial and strictly decreasing.  Then
		\begin{equation}\label{key-lema2}
		\int_{\R^N\setminus\Omega}f(x)\ dx>\int_{\R^N\setminus\Omega^*} f(x)\ dx~~~\text{if}~~|\Omega^*\Delta\Omega|>0.
		\end{equation}
	\end{itemize}
\end{lemma}
We now give the proof of Theorem \ref{main-result}.

\begin{proof}[Proof of Theorem \ref{main-result}]
	The proof is inspired by that of Theorem 1.1 in \cite{li2019symmetric}.  We start with the case when $\Omega$ is not a ball. Note that $|\Omega^*|=|\Omega|$ and $|\Omega\setminus\Omega^*|=|\Omega^*\setminus\Omega|=\frac{1}{2}|\Omega^*\Delta\Omega|>0$ since $\Omega$ is not a ball. 
	
	Since $\Omega$ is an open set, we may fix $R_0>0$ such that  $B_{R_0}:=B_{R_0}(x_0)\subset\Omega$ for some $x_0\in\Omega$. Now, let  $\eta\in C^{\infty}_c(\Omega)$ with $\supp\eta\subset B_{R_0}$ be a nonnegative radial decreasing function such that $\eta=1$ for $|x-x_0|\leq\frac{R_0}{2}$. For all $\varepsilon\in(0,1)$, we define 
	\begin{equation*}
	u_{\varepsilon}(x)=\eta\Big(\frac{x}{\varepsilon}\Big).
	\end{equation*}
	Then, from the properties of $\eta$, it is easy to see that
	\begin{equation*}
	u^*_{\varepsilon}(x)=\eta\Big(\frac{x}{\varepsilon}\Big).
	\end{equation*}
	Now,
	\begin{align}\label{e8}
	\nonumber&\iint_{\Omega\times \Omega}G\Big(\frac{|u_{\varepsilon}(x)-u_{\varepsilon}(y)|}{|x-y|^s}\Big)\frac{dxdy}{|x-y|^N}\\
	\nonumber&=\iint_{\R^N\times\R^N}G\Big(\frac{|u_{\varepsilon}(x)-u_{\varepsilon}(y)|}{|x-y|^s}\Big)\frac{dxdy}{|x-y|^N}-2\int_{\Omega}\int_{\R^N\setminus \Omega}G\Big(\frac{|u_{\varepsilon}(x)|}{|x-y|^s}\Big)\frac{dydx}{|x-y|^N}\\
	&=\iint_{\R^N\times\R^N}G\Big(\frac{|u^*_{\varepsilon}(x)-u^*_{\varepsilon}(y)|}{|x-y|^s}\Big)\frac{dxdy}{|x-y|^N}-2\int_{\Omega}\int_{\R^N\setminus \Omega}G\Big(\frac{|u_{\varepsilon}(x)|}{|x-y|^s}\Big)\frac{dydx}{|x-y|^N} 
	\end{align}
	So, the proof is completed if we show that for $\varepsilon$ sufficiently small,
	\begin{align}\label{b0}
	\int_{\Omega}\int_{\R^N\setminus \Omega}G\Big(\frac{|u_{\varepsilon}(x)|}{|x-y|^s}\Big)\frac{dydx}{|x-y|^N}>\int_{\Omega^*}\int_{\R^N\setminus \Omega^*}G\Big(\frac{|u_{\varepsilon}^*(x)|}{|x-y|^s}\Big)\frac{dydx}{|x-y|^N}
	\end{align}
	that is
	\begin{equation*}
	\int_{B_{R_0}}\int_{\R^N\setminus \Omega}G\Big(\frac{|\eta(x)|}{|\varepsilon x-y|^s}\Big)\frac{dydx}{|\varepsilon x-y|^N}>	\int_{B_{R_0}}\int_{\R^N\setminus \Omega^*}G\Big(\frac{|\eta(x)|}{|\varepsilon x-y|^s}\Big)\frac{dydx}{|\varepsilon x-y|^N},
	\end{equation*}
	that is
	\begin{equation}\label{b11}
	\int_{B_{R_0}}H_{\Omega}(\varepsilon x)\ dx>\int_{B_{R_0}}H_{\Omega^*}(\varepsilon x)\ dx, 
	\end{equation}
	where
	\begin{align*}
	H_{\Omega}(\varepsilon x)=\int_{\R^N\setminus \Omega}G\Big(\frac{|\eta(x)|}{|\varepsilon x-y|^s}\Big)\frac{dy}{|\varepsilon x-y|^N}~~\text{and}~~H_{\Omega^*}(\varepsilon x)=\int_{\R^N\setminus \Omega^*}G\Big(\frac{|\eta(x)|}{|\varepsilon x-y|^s}\Big)\frac{dy}{|\varepsilon x-y|^N}.
	\end{align*}
	Note that for $x\in B_{R_0}$, the function $y\mapsto h(x,y)=|y|^{-N}G\Big(\frac{|\eta(x)|}{|y|^s}\Big)$ is radial and strictly decreasing. Moreover, $h(x,\cdot)\in L^1(\R^N\setminus\Omega)$. Thus, from \eqref{key-lema2},
	\begin{equation*}
	H_{\Omega}(0)>H_{\Omega^*}(0).
	\end{equation*}
	So for $\varepsilon$ sufficiently small, \eqref{b11} holds true by dominated convergence theorem and thus \eqref{b0} follows. 
	
	Using this in \eqref{e8}, we get 
		\begin{align*}
	\nonumber&\iint_{\Omega\times \Omega}G\Big(\frac{|u_{\varepsilon}(x)-u_{\varepsilon}(y)|}{|x-y|^s}\Big)\frac{dxdy}{|x-y|^N}\\
	&<\iint_{\R^N\times\R^N}G\Big(\frac{|u^*_{\varepsilon}(x)-u^*_{\varepsilon}(y)|}{|x-y|^s}\Big)\frac{dxdy}{|x-y|^N}-2\int_{\Omega^*}\int_{\R^N\setminus \Omega^*}G\Big(\frac{|u_{\varepsilon}^*(x)|}{|x-y|^s}\Big)\frac{dxdy}{|x-y|^N} \\	
	&=\iint_{\Omega^*\times \Omega^*}G\Big(\frac{|u_{\varepsilon}^*(x)-u_{\varepsilon}^*(y)|}{|x-y|^s}\Big)\frac{dydx}{|x-y|^N},
	\end{align*}
	as wanted. 
	
	We now focus in the case when $\Omega$ is a ball, say $\Omega=B_1$. As above, we pick a nonnegative radial decreasing function $\eta\in C^{\infty}_c(B_1)$ and consider $\tilde{x}$ with $|\tilde{x}|=\frac{1}{2}$. Next, we define 
	\begin{equation*}
	u_{\varepsilon}(x)=\eta\Big(\frac{x-\tilde{x}}{\varepsilon}\Big),~~~x\in B_1.
	\end{equation*}
	Then from the definition of $\eta$, we have
	\begin{equation*}
	u_{\varepsilon}^*(x)=\eta\Big(\frac{x}{\varepsilon}\Big). 
	\end{equation*}
	To complete the proof, we only need to check \eqref{b0}. Note that, since $\Omega$ is a ball, then $\Omega^*=\Omega$. Therefore $H_{\Omega}=H_{\Omega^*}$. So, it is sufficient to verify that
	\begin{equation}\label{w}
	\int_{B_1}H_{B_1}(\tilde{x}+\varepsilon x)\ dx>\int_{B_1}H_{B_1}(\varepsilon x)\ dx.
	\end{equation}
	Now, 
	\begin{align*}
	H_{B_1}(\tilde{x})=\int_{\R^N\setminus B_1}G\Big(\frac{|\eta(x)|}{|\tilde{x}-y|^s}\Big)\frac{dy}{|\tilde{x}-y|^N}&=\int_{\R^N\setminus B_1(\tilde{x})}G\Big(\frac{|\eta(x)|}{|z|^s}\Big)\frac{dz}{|z|^N}\\
	&>\int_{\R^N\setminus B_1}G\Big(\frac{|\eta(x)|}{|z|^s}\Big)\frac{dz}{|z|^N}=H_{B_1}(0).
	\end{align*}
	Note that in the above inequality, we have used \eqref{key-lema2} and the fact that $(B_1(\tilde{x}))^*=B_1$. Therefore, for $\varepsilon$ sufficiently small, \eqref{w} holds, thanks to dominated convergence theorem. This completes the proof.
\end{proof}

\begin{proof}[Proof of Theorem \ref{second-main-result}]
	Let $u\in C^{\infty}_c(\Omega)$ be a nonnegative function. Then
	\begin{align}\label{z1}
\nonumber	\iint_{\R^N\times\R^N}G\Big(&\frac{|u^*(x)-u^*(y)|}{|x-y|^s}\Big)\frac{dxdy}{|x-y|^N}\leq \iint_{\R^N\times\R^N}G\Big(\frac{|u(x)-u(y)|}{|x-y|^s}\Big)\frac{dxdy}{|x-y|^N}\\
	&=\iint_{\Omega\times\Omega}G\Big(\frac{|u(x)-u(y)|}{|x-y|^s}\Big)\frac{dxdy}{|x-y|^N}+2\int_{\Omega}\int_{\R^N\setminus\Omega}G\Big(\frac{|u(x)|}{|x-y|^s}\Big)\frac{dydx}{|x-y|^N}.
	\end{align}
	Notice that in the first inequality, we have used the P\'{o}lya-Szeg\"{o} inequality \eqref{polya-szego-orlicz-rn} with $M(t)=t^s$ and $N(t)=t^N$. 
	
	We now focus on the second term on the right-hand side of \eqref{z1}. In the sequel, we denote by $\delta(x)=\dist(x,\partial\Omega)$ the distance to the boundary $\partial\Omega$ of $\Omega$. Then, using Lemma \ref{two-side-g}, we get 
	\begin{align}\label{z2}
\nonumber	\int_{\Omega}\int_{\R^N\setminus\Omega}G\Big(\frac{|u(x)|}{|x-y|^s}\Big)\frac{dydx}{|x-y|^N}&=\int_{\Omega}\int_{\R^N\setminus\Omega}G\Big(\frac{|u(x)|}{\delta(x)^s}\frac{\delta(x)^s}{|x-y|^s}\Big)\frac{dydx}{|x-y|^N}\\
	&\leq\int_{\Omega}G\Big(\frac{|u(x)|}{\delta(x)^s}\Big)\delta(x)^{sp^-_G}\ dx\int_{\R^N\setminus\Omega}\frac{dy}{|x-y|^{N+sp^-_G}}.
	\end{align}
 Using that\footnote{The notation $f\asymp g$ means that the two-sided estimate $c_1g\leq f\leq c_2g$ is true, where $c_1, c_2$ are two positive constants.} (see Eq. (1,3,2,12) in \cite{grisvard2011elliptic})
	\begin{equation*}
	\int_{\R^N\setminus\Omega}\frac{dy}{|x-y|^{N+sp^-_G}}\asymp\delta(x)^{-sp^-_G},~~~~\forall x\in\Omega,
	\end{equation*}
	it follows from \eqref{z2} that
	\begin{align}\label{z3}
	\nonumber	\int_{\Omega}\int_{\R^N\setminus\Omega}G\Big(\frac{|u(x)|}{|x-y|^s}\Big)\frac{dydx}{|x-y|^N}&\leq C(N,s,G)\int_{\Omega}G\Big(\frac{|u(x)|}{\delta(x)^s}\Big)\ dx\\
		&\leq C(N,s,G,\Omega)\iint_{\Omega\times\Omega}G\Big(\frac{|u(x)-u(y)|}{|x-y|^s}\Big)\frac{dxdy}{|x-y|^N}, 
	\end{align}
	where in the latter, we have used the fractional Orlicz-Hardy inequality, see \cite[Theorems 1.5 \& 1.2]{bal2022hardy}. Combining \eqref{z3} and \eqref{z1} we obtain 
	\begin{equation}\label{z4}
	\iint_{\R^N\times\R^N}G\Big(\frac{|u^*(x)-u^*(y)|}{|x-y|^s}\Big)\frac{dxdy}{|x-y|^N}\leq C(N,s,G,\Omega)\iint_{\Omega\times\Omega}G\Big(\frac{|u(x)-u(y)|}{|x-y|^s}\Big)\frac{dxdy}{|x-y|^N}.
	\end{equation}
	This yields \eqref{n1}. Finally, the estimate \eqref{n2} follows from \eqref{z4} and the fact that 
	\begin{equation*}
	\iint_{\Omega^*\times\Omega^*}G\Big(\frac{|u^*(x)-u^*(y)|}{|x-y|^s}\Big)\frac{dxdy}{|x-y|^N}\leq \iint_{\R^N\times\R^N}G\Big(\frac{|u^*(x)-u^*(y)|}{|x-y|^s}\Big)\frac{dxdy}{|x-y|^N}. 
	\end{equation*}
\end{proof}

\section{Extensions and applications}\label{section:application}
In this section, we present some examples of Young functions that satisfy the assumptions of our main results. We start with the power function $G(t)=t^p$, with $p\geq 2$, in which case Theorems \ref{main-result} and \ref{second-main-result} coincide with \cite[Theorems 1.1 and 1.2]{li2019symmetric}. Thus, our results can be considered as an extension or generalization of those in \cite{li2019symmetric} to the nonhomogeneous situation. 

Our results also apply to the following prototype of Young functions with a more general growth
\begin{equation*}
G(t)=t^p(1+|\log t|)~~~\text{and}~~~G(t)=\frac{t^p}{\log(e+t)},~~p\geq2.
\end{equation*}
Furthermore, as mentioned in the introduction, Young functions of the form
\begin{equation*}
G(t)=t^q+t^p~~\text{with}~~ p>q\geq 2,
\end{equation*}
also fall into the class studied in this paper. This structure is closely related to the so-called \textit{double phase variational problems}, see e.g. \cite{colombo2015regularity} and the references therein.

\section*{Data availability statement}
Data sharing not applicable to this article as no datasets were generated or analyzed during the current study.

\section*{Declaration of competing interest}
The author has no competing interests to declare that are relevant to the content of this article.\\
~

\textbf{Acknowledgements:}  The author is supported by Fields Institute. Part of this work has been carried out during a visit of the author to the Department of Mathematics at Rutgers University (USA) in October $2024$, as part of the Abbas Bahri Excellence Fellowship for Graduate and Postdoctoral Research in Mathematics. He wishes to thank the Department for their warm hospitality.


\bibliographystyle{ieeetr}

\begin{thebibliography}{10}
	


\bibitem{almgren1989symmetric} F. J. Almgren Jr and E. H. Lieb, \emph{Symmetric decreasing rearrangement is sometimes continuous.} Journal of the American Mathematical Society (1989): 683-773. 
	
\bibitem{bahrouni2021neumann} S. Bahrouni and A. M. Salort, \emph{Neumann and Robin type boundary conditions in Fractional Orlicz-Sobolev spaces.} ESAIM: Control, Optimisation and Calculus of Variations 27 (2021): S15.

\bibitem{bal2022hardy} K. Bal, K. Mohanta, P. Roy, and F. Sk, \emph{Hardy and Poincar\'{e} inequalities in fractional Orlicz-Sobolev spaces.} Nonlinear Analysis 216 (2022): 112697.

\bibitem{bonder2019fractional} F. J. Bonder and A. M. Salort, \emph{Fractional order orlicz-sobolev spaces.} Journal of Functional Analysis 277.2 (2019): 333-367.

\bibitem{colombo2015regularity} M. Colombo and G. Mingione, \emph{Regularity for double phase variational problems.} Archive for Rational Mechanics and Analysis 215 (2015): 443-496.

\bibitem{de2021polya} P. De N\'{a}poli, J. Fernandez Bonder, and A. Salort, \emph{A P\'{o}lya-Szeg\"{o} principle for general fractional Orlicz-Sobolev spaces.} Complex Variables and Elliptic Equations 66.4 (2021): 546-568.

\bibitem{frank2008non} R. L. Frank and R. Seiringer, \emph{Non-linear ground state representations and sharp Hardy inequalities.} Journal of Functional Analysis 255.12 (2008): 3407-3430.

\bibitem{grisvard2011elliptic} P. Grisvard, \emph{Elliptic problems in nonsmooth domains.} Society for Industrial and Applied Mathematics, 2011. 

\bibitem{kranosel1961convex} M. A. Krasnosel'ski\u{i} and Ja. Rutickii, \emph{Convex functions and orlicz spaces}, 1961, Translated from the first Russian edition by Leo F. Boron, P. Noordhoff Ltd. Groningen.

\bibitem{kufner1979function} A. Kufner, O. John, and S. Fucik, \emph{Function Spaces.}, Vol. 3. Springer Science Business Media (1979).

\bibitem{li2019symmetric} D. Li and K. Wang, \emph{Symmetric radial decreasing rearrangement can increase the fractional Gagliardo norm in domains.} Communications in Contemporary Mathematics 21.07 (2019): 1850059.

\bibitem{lieb1983sharp} E. H. Lieb, \emph{Sharp constants in the Hardy-Littlewood-Sobolev and related inequalities.} Annals of Mathematics (1983): 349-374.

\bibitem{polya1951isoperimetric} G. P\'{o}lya and G. Szeg\"{o}, \emph{Isoperimetric Inequality in Mathematical Physics.} Annals of Mathematics Studies, no. 27, Princeton University Press, Princeton, N.J., 1951.

\bibitem{salort2020eigenvalues} A. M. Salort, \emph{Eigenvalues and minimizers for a non-standard growth non-local operator.} Journal of Differential Equations 268.9 (2020): 5413-5439.

\bibitem{talenti1976best} G. Talenti, \emph{Best constant in Sobolev inequality.} Annali di Matematica pura ed Applicata 110 (1976): 353-372.

\bibitem{talenti1994inequalities} G. Talenti, \emph{Inequalities in rearrangement invariant function spaces.} Nonlinear analysis, function spaces and applications (1994): 177-230.

	
	
\end{thebibliography}

\end{document}